\documentclass[a4paper,11pt]{article}
\usepackage[latin1]{inputenc}
\usepackage[english]{babel}
\usepackage{amsmath}
\usepackage{amsfonts}
\usepackage{amssymb}
\usepackage{epsfig}
\usepackage{amsopn}
\usepackage{amsthm}
\usepackage{color}
\usepackage{graphicx}
\usepackage{enumerate}
\usepackage{mathrsfs}
\usepackage{cite}
\newtheorem{theorem}{Theorem}[section]
\newtheorem{corollary}[theorem]{Corollary}
\newtheorem{lemma}[theorem]{Lemma}

\newtheorem{definition}[theorem]{Definition}

\newtheorem*{theorem*}{Theorem}
\newtheorem*{lemma*}{Lemma}
\newtheorem*{remark*}{Remark}
\newtheorem*{definition*}{Definition}
\newtheorem*{proposition*}{Proposition}
\newtheorem*{corollary*}{Corollary}
\numberwithin{equation}{section}
\newcommand{\real}{\mathbb{R}}

\let\ced=\c         

\def\qed{\,\unskip\kern 6pt \penalty 500
\raise -2pt\hbox{\vrule \vbox to8pt{\hrule width 6pt
\vfill\hrule}\vrule}\par}
\definecolor{darkblue}{rgb}{0.05, .05, .65}
\definecolor{darkgreen}{rgb}{0.1, .65, .1}
\definecolor{darkred}{rgb}{0.8,0,0}
\newcommand{\beqn}{\begin{equation}}
\newcommand{\eeqn}{\end{equation}}
\newcommand{\bear}{\begin{eqnarray}}
\newcommand{\eear}{\end{eqnarray}}
\newcommand{\bean}{\begin{eqnarray*}}
\newcommand{\eean}{\end{eqnarray*}}
\begin{document}

\title{\huge \bf Finite time extinction for a diffusion equation with spatially inhomogeneous strong absorption}

\author{
\Large Razvan Gabriel Iagar\,\footnote{Departamento de Matem\'{a}tica
Aplicada, Ciencia e Ingenieria de los Materiales y Tecnologia
Electr\'onica, Universidad Rey Juan Carlos, M\'{o}stoles,
28933, Madrid, Spain, \textit{e-mail:} razvan.iagar@urjc.es}
\\[4pt] \Large Philippe Lauren\ced{c}ot\,\footnote{Institut de
Math\'ematiques de Toulouse, CNRS UMR~5219, Universit\'e Paul Sabatier, F--31062 Toulouse Cedex 9, France. \textit{e-mail:}
Philippe.Laurencot@math.univ-toulouse.fr}\\ [4pt] }
\date{}
\maketitle

\begin{abstract}
The phenomenon of finite time extinction of bounded and non-negative solutions to the diffusion equation with strong absorption
$$
\partial_t u-\Delta u^m+|x|^{\sigma}u^q=0, \qquad (t,x)\in(0,\infty)\times\real^N,
$$
with $m\geq1$, $q\in(0,1)$ and $\sigma>0$, is addressed. Introducing the critical exponent $\sigma^* := 2(1-q)/(m-1)$ for $m>1$ and $\sigma_*=\infty$ for $m=1$, extinction in finite time is known to take place for $\sigma\in [0,\sigma^*)$ and an alternative proof is provided therein. When $m>1$ and $\sigma\ge \sigma^*$, the occurrence of finite time extinction is proved for a specific class of initial conditions, thereby supplementing results on non-extinction that are available in that range of $\sigma$ and showing their sharpness.
\end{abstract}

\bigskip


\smallskip


\section{Introduction}

The aim of this short note is to address the question of finite time extinction of bounded and non-negative solutions to the following diffusion equation with spatially inhomogeneous strong absorption
\begin{equation}\label{eq1}
\partial_tu=\Delta u^m-|x|^{\sigma}u^q, \qquad (t,x)\in(0,\infty)\times\real^N,
\end{equation}
with initial condition
\begin{equation}\label{init.cond}
	u(0)=u_0\in L_+^{\infty}(\real^N) := \big\{ z\in L^\infty(\real^N)\ :\ z(x)\ge 0 \;\text{ a.e. in }\; \real^N \big\},
\end{equation}
where
\begin{equation}\label{range.exp}
m\geq1, \qquad 0<q<1, \qquad \sigma>0.
\end{equation}
The dynamics of Eq.~\eqref{eq1} features a double competition: on the one hand, the effects of the diffusion, which spreads the mass of the solutions as time advances, competes with the absorption term, which triggers a loss of mass, possibly leading to vanishing in finite time. On the other hand, the weight on the absorption term is likely to bring into play significant imbalances between the properties of the solutions in a neighborhood of the origin (where, at least formally, the absorption term is very small) and in outer regions, at uniform distance from the origin, where the weight $|x|^{\sigma}$ becomes very strong.

With respect to the first of the two competitions described above, the spatially homogeneous counterpart of Eq.~\eqref{eq1} with $\sigma=0$, that is
\begin{equation}\label{eq1.hom}
\partial_tu-\Delta u^m+u^q=0, \qquad m\geq1,
\end{equation}
has a well understood dynamics at least in some ranges of exponents. Indeed, the critical values for the absorption exponent $q$ are $q=m$ and $q=1$, and its solutions have very different properties according to whether $q>m$, $1<q\leq m$ or $0<q<1$. The former of these three ranges sees the diffusion term being sufficiently strong either to govern the dynamics, or to balance the effects of the absorption giving rise to profiles known as very singular solutions. Properties such as decay estimates for the solutions, construction of singular (or very singular) self-similar solutions and large time behavior of solutions in the sense of convergence to such self-similar profiles as $t\to\infty$ are established in a number of works, see for example \cite{KP86, PT86, KU87, KV88, KPV89, Le97, Kwak98} and references therein. While for $q>m$ all compactly supported solutions have an algebraic time decay as $t\to\infty$ and their supports expand reaching the whole space in the limit, a different situation occurs for $m>1$ and $q\in(1,m)$. In the latter range, the effect of the absorption is dominant and the expansion of the positivity region of the compactly supported solutions is limited, leading to the \emph{localization of supports}, that is, there exists a large radius $R>0$ not depending on time such that ${\rm supp}\,u(t)\subseteq B(0,R)$ for any $t>0$. Solutions still present an algebraic decay as $t\to\infty$, but self-similar solutions might become unbounded, presenting a specific growth at infinity \cite{MPV91, CVW97} and delicate descriptions of the large time behavior in the form of matched asymptotics between ``flat" solutions of the form $K_*t^{-1/(q-1)}$ for some explicit constant $K_*>0$, and boundary layers appearing near the boundary of the localization ball $B(0,R)$, have been established, see \cite{CV99}. Working with spatially inhomogeneous absorption with general weights, Peletier and Tesei established in dimension $N=1$ in \cite{PT85,PT86b} positivity of supports for $q>m$, conditions for localization of supports of the solutions for $1<q<m$ and the existence of stationary compactly supported solutions also for $1<q<m$.

The dynamics of both Eq.~\eqref{eq1.hom} and Eq.~\eqref{eq1} seem to be by far more involved in the range $m\geq1$ and $0<q<1$, also known as the \emph{strong absorption range} due to the fact that the absorption term prevails. This dominance gives rise to two new mathematical phenomena not present for $q\geq1$. On the one hand, \emph{finite time extinction} of (non-negative bounded) solutions occurs. This means that there exists a time $T_e\in(0,\infty)$ such that $u(t)\not\equiv0$ for $t\in(0,T_e)$ but $u(T_e)\equiv0$, $T_e$ being thus called the extinction time of $u$. The finite time extinction stems from the ordinary differential equation $\partial_t u=u^q$ obtained by neglecting the diffusion, emphasizing thus the strength of the absorption, see \cite{Ka75, Ka84}. On the other hand, \emph{instantaneous shrinking of supports} of solutions to Eq.~\eqref{eq1.hom} with bounded initial condition $u_0$ such that $u_0(x)\to0$ as $|x|\to\infty$ takes place, that is, for any non-negative initial condition $u_0\in L^{\infty}(\real^N)$ such that $u_0(x)\to0$ as $|x|\to\infty$ and $\tau>0$, there is $R(\tau)>0$ such that ${\rm supp}\,u(t)\subseteq B(0,R(\tau))$ for all $t\ge\tau$. This is once more due to the strength of the absorption term, which involves a very quick loss of mass, and has been proved in \cite{EK79, Ka84} in the semilinear case $m=1$ and in \cite{Abd98} for $m>1$. Finer properties of the dynamics of Eq.~\eqref{eq1.hom} in this range are still lacking in general. For example, a description of the extinction rates and behavior near the extinction time of the solutions to Eq.~\eqref{eq1.hom} seems to be only available when $m+q=2$ in \cite{GV94}, revealing a case of asymptotic simplification, and appears to be a complicated problem if $m+q\neq2$.

Let us now turn our attention to Eq.~\eqref{eq1} with exponents as in \eqref{range.exp} and introduce the critical exponent
\begin{equation}\label{crit.exp}
\sigma^*:=\left\{\begin{array}{ll}
	\displaystyle{\frac{2(1-q)}{m-1}}>0, & {\rm if} \ m>1, \\
	 & \\
	 \infty, & {\rm if} \ m=1.\end{array}\right.
\end{equation}
It is established in \cite{BHV01, KV97} ($m=1$) and in \cite{Belaud01,BeSh07} ($m>1$) that any solution to Eq.~\eqref{eq1} posed in a bounded domain $\Omega\subset\real^N$ with homogeneous Neumann boundary condition vanishes in finite time provided $0<\sigma<\sigma^*$ (the analysis performed in the above mentioned references actually deals with more general weights instead of $|x|^\sigma$). A similar result is shown in \cite{BHV01, BD10} ($m=1$) and in \cite{BeSh22} ($m>1$) for homogeneous Dirichlet boundary conditions. A direct consequence of the latter and the localization of supports established in \cite[Theorem~1.1]{ILS22} is that bounded and non-negative weak solutions to~\eqref{eq1}-\eqref{init.cond} vanish identically after a finite time when $\sigma\in [0,\sigma^*)$. We shall recall this result in Theorem~\ref{th.1} below and provide an alternative proof, relying on integral inequalities and the well-known $L^1-L^\infty$-regularizing effect of the heat equation or the porous medium equation in $\real^N$. In the complementary range $m>1$ and $\sigma\ge \sigma^*$, fewer results seem to be available but finite time extinction is not a generic feature. Indeed, we recently proved in \cite{ILS22} that, when $m>1$, $\sigma>\sigma^*$, and the initial condition $u_0$ is positive in a neighborhood of the origin, the positivity set
\begin{equation*}
	\mathcal{P}(t) := \{ x\in\real^N\ :\ u(t,x)>0\}
\end{equation*}
of the solution $u$ to \eqref{eq1}-\eqref{init.cond} is non-empty and contains the origin for all $t\ge 0$. Moreover, $\mathcal{P}(t)$ shrinks to $\{0\}$ as $t\to\infty$. We also identify in \cite{ILS22} a class of initial conditions $u_0$ for which $0\not\in\mathcal{P}(t)$ for all $t\ge 0$; that is, $u(t)$ vanishes at the origin where the absorption is the weakest. It is then tempting to figure out whether finite time extinction could occur and the purpose of this note is to answer this question by the affirmative.

\medskip

\noindent \textbf{Main results}. We deal in this note with the Cauchy problem~\eqref{eq1}~-\eqref{init.cond} and first make precise the notion of solution we are using in the present work.
\begin{definition}\label{def.wp}
A non-negative weak solution to the Cauchy problem \eqref{eq1}-\eqref{init.cond} is a function
\begin{subequations}\label{wp4}
\begin{equation}
	u \in L_+^\infty((0,\infty)\times\real^N) \label{wp4a}
\end{equation}
such that, for all $T>0$,
\begin{equation}
	u^m \in L^2\big((0,T),H^1_{{\rm loc}}(\real^N)\big) \label{wp4b}
\end{equation}
and
\begin{equation}
	\int_0^T \int_{\real^N} \Big[ (u_0-u) \partial_t \zeta + \nabla u^m \cdot \nabla\zeta + |x|^\sigma u^q \zeta \Big]\ dxds = 0 \label{wp4c}
\end{equation}
for all $\zeta\in C_c^1([0,T)\times\real^N)$.
\end{subequations}
\end{definition}
Basic results of well-posedness of the Cauchy problem and of instantaneous shrinking and localization of supports of solutions in the framework of Definition~\ref{def.wp} are established in \cite[Theorem~1.1]{ILS22}, the latter being restated as a preliminary fact at the beginning of Section~\ref{sec.th1}. Let us next recall that finite time extinction \emph{always occurs} if $\sigma<\sigma^*$, see \cite[Theorem~3.1]{BHV01} and \cite[Theorem~3.16]{BD10} ($m=1$) and \cite[Theorem~2.1]{BeSh22} ($m>1$).

\begin{theorem}\label{th.1}
If $m$ and $q$ are as in \eqref{range.exp}, $u_0$ is as in \eqref{init.cond} and $\sigma<\sigma^*$, then the solution $u$ to the Cauchy problem \eqref{eq1}-\eqref{init.cond} vanishes in finite time.
\end{theorem}

As already pointed out, Theorem~\ref{th.1} follows from the above mentioned references and the property of localization of supports of any solution to the Cauchy problem~\eqref{eq1}-\eqref{init.cond}. Nevertheless, we give a short and much simpler proof based on estimates on the $L^1$ norm of the solution at different positive times, which also works for the spatially homogeneous equation~\eqref{eq1.hom}. Let us also stress here that, in the semilinear case $m=1$, $\sigma^*=\infty$ and all bounded and non-negative solutions vanish in finite time.

However, things are a bit more complex in the range $m>1$ and $\sigma\geq\sigma^*$. Indeed, for $\sigma>\sigma^*$, it is proved in \cite[Theorem~1.3]{ILS22} that solutions to the Cauchy problem~\eqref{eq1}-\eqref{init.cond} with an initial condition $u_0$ which is positive in a small ball $B(0,r)$ converge as $t\to\infty$ to a unique non-zero self-similar solution and in particular do no longer vanish in finite time. We thus infer that we need to \emph{restrict the class of initial conditions} $u_0$ for the finite time extinction to hold true. This is made precise in the following statement, which is the main result of this note.

\begin{theorem}\label{th.2}
Let $m>1$, $0<q<1$ and $\sigma\geq\sigma^*$. Consider $a$, $A$ and $R>0$ such that $a>\sigma/(1-q)$ and
\begin{equation}\label{interm1}
A^{m-q}R^{a(m-q)-\sigma-2}\leq\frac{1}{am(am+N-2)}.
\end{equation}
Let $u_0\in L_+^{\infty}(\real^N)$ be such that
\begin{subequations}\label{cond.u0}
\begin{equation}
	u_0(x)\leq A|x|^{a}, \qquad {\rm for \ any} \ x\in B(0,R), \label{cond.u0a}
\end{equation}
and
\begin{equation}
	\|u_0\|_{\infty}\leq AR^a. \label{cond.u0b}
\end{equation}
\end{subequations}
Then the solution $u$ to \eqref{eq1}-\eqref{init.cond} vanishes in finite time.
\end{theorem}

According to our previous results in \cite{ILS22}, the flatness condition~\eqref{cond.u0a} as $|x|\to0$ is required in order to have finite time extinction. Moreover, in a forthcoming work \cite{ILS22b} we will show that a limitation on $\|u_0\|_{\infty}$ such as the one in \eqref{cond.u0b} is also needed, as a self-similar solution with exponential time decay as $t\to\infty$ and a dead core (that is, support located in an annulus far away from the origin) can be constructed for $\sigma=\sigma^*$. We may thus conclude that our results are qualitatively sharp. A corollary of Theorem~\ref{th.2} which provides examples of initial conditions $u_0$ to which it applies is given in Section~\ref{sec.th2} after its proof.

\section{Proof of Theorem \ref{th.1}}\label{sec.th1}

This section is dedicated to the proof of Theorem~\ref{th.1}. Before the beginning of the proof, we recall here as a preliminary the precise statement of the well-posedness theorem for Eq.~\eqref{eq1}, which can be found in \cite[Theorem~1.1]{ILS22}. It includes properties such as instantaneous shrinking, localization of supports of bounded solutions and the comparison principle, that will be used in the sequel.
\begin{theorem}\label{th.wp}
For any $m>1$, $q\in(0,1)$ and $\sigma>0$, there is a unique non-negative weak solution to the Cauchy problem~\eqref{eq1}, \eqref{init.cond} which satisfies
\begin{equation}
		\|u(t)\|_\infty \le \|u_0\|_\infty\,, \qquad t\ge 0. \label{wp0}
\end{equation}
In addition, it enjoys the properties of \emph{instantaneous shrinking} and \emph{localization} of the support; that is, for any $t>0$, $u(t)$ has compact support and, given $\tau>0$, there exists $R=R(\tau)>0$, depending on $u_0$ and $\tau$ but not on $t\in [\tau,\infty)$, such that
	$$
	{\rm supp}\,u(t)\subseteq B(0,R(\tau)), \qquad {\rm for \ any} \ t\geq\tau.
	$$
Also, the following \emph{comparison principle} holds true: given $u_{0,i}\in L_+^\infty(\real^N)$, $i=1,2$, such that $u_{0,1}\le u_{0,2}$ in $\real^N$, the corresponding non-negative weak solutions $u_1$ and $u_2$ to~\eqref{eq1}, \eqref{init.cond} satisfy $u_1\le u_2$ in $(0,\infty)\times\real^N$.
\end{theorem}
For a proof, we refer the reader to \cite[Sections~2 and~3]{ILS22}, while similar results for the semilinear case $m=1$ follow by a simple adaptation of the proofs. With these preparations, we are ready to prove our first main result. This is done by adapting an argument from \cite{BLS02}.
\begin{proof}[Proof of Theorem~\ref{th.1}]
According to Theorem~\ref{th.wp}, we may assume without loss of generality that $u_0$ is compactly supported and that there exists $R>0$ such that
\begin{equation}\label{loc}
{\rm supp}\,u(t)\subset B(0,R), \qquad {\rm for \ any} \ t>0.
\end{equation}
Owing to the non-negativity of $u$, it follows from Eq.~\eqref{eq1} that for any $T>0$ and $t\in(0,T)$ we have
\begin{equation}\label{interm2}
\int_{t}^T\int_{\real^N}|x|^{\sigma}u^q(s,x)\,dx\,ds\leq \|u(T)\|_{1}+\int_{t}^T\int_{\real^N}|x|^{\sigma}u^q(s,x)\,dx\,ds\leq \|u(t)\|_1.
\end{equation}
Pick now $b\in\real$ such that
\begin{equation}\label{interm3}
\frac{\sigma}{N}+1-q<b<\frac{\sigma^*}{N}+1-q.
\end{equation}
We then infer from \eqref{interm2} that
\begin{equation}\label{interm4}
\begin{split}
\|u(t)\|_1&\geq\int_t^T\|u(s)\|_{\infty}^{-b}\int_{\real^N}|x|^{\sigma}\|u(s)\|_{\infty}^bu^q(s,x)\,dx\,ds\\
&\geq\int_t^T\|u(s)\|_{\infty}^{-b}\int_{\real^N}|x|^{\sigma}u^{b+q}(s,x)\,dx\,ds.
\end{split}
\end{equation}
Next, on the one hand we deduce from \eqref{loc}, H\"{o}lder's inequality and the non-negativity of $u$ that, for any $s\in(t,T)$
\begin{equation*}
\begin{split}
\|u(s)\|_1&=\int_{B(0,R)}|x|^{\sigma/(b+q)}u(s,x)|x|^{-\sigma/(b+q)}\,dx\\
&\leq\left(\int_{B(0,R)}|x|^{\sigma}u^{b+q}(s,x)\,dx\right)^{1/(b+q)}\left(\int_{B(0,R)}|x|^{-\sigma/(b+q-1)}\,dx\right)^{(b+q-1)/(b+q)}\\
&\leq C\left(\int_{\real^N}|x|^{\sigma}u^{b+q}(s,x)\,dx\right)^{1/(b+q)},
\end{split}
\end{equation*}
since \eqref{interm3} guarantees that $\sigma/(b+q-1)<N$. We thus find that
\begin{equation}\label{interm5}
\|u(s)\|_1^{b+q}\leq C\left(\int_{\real^N}|x|^{\sigma}u^{b+q}(s,x)\,dx\right).
\end{equation}
On the other hand, since $u$ is a subsolution to the porous medium equation if $m>1$ or to the heat equation if $m=1$, it follows from the well-known regularizing effect of the porous medium equation (see for example \cite[Theorem~2.1]{VazSmooth}) if $m>1$ or from the standard representation formula for the heat equation if $m=1$ that
$$
\|u(s)\|_{\infty}\leq C(s-t)^{-\theta}\|u(t)\|_1^{2\theta/N}, \qquad s\geq t,
$$
where $\theta=N/(N(m-1)+2)$. Consequently, taking into account that $b>0$, we have
\begin{equation}\label{interm6}
\|u(s)\|_{\infty}^{-b}\geq C(s-t)^{b\theta}\|u(t)\|_1^{-2b\theta/N}.
\end{equation}
Combining the estimates~\eqref{interm4}, \eqref{interm5} and~\eqref{interm6} gives, for any $s\in(t,T)$, that
$$
\|u(t)\|_1\geq C\int_t^{T}(s-t)^{b\theta}\|u(t)\|_{1}^{-2b\theta/N}\|u(s)\|_1^{b+q}\,ds,
$$
or equivalently
\begin{equation}\label{interm7}
\int_t^{T}(s-t)^{b\theta}\|u(s)\|_1^{b+q}\,ds\leq C\|u(t)\|_1^{(N+2b\theta)/N}.
\end{equation}
Since $\|u(s)\|_1\geq\|u(T)\|_1$ for $s\in(t,T)$ by~\eqref{eq1}, we further infer from \eqref{interm7} that
$$
C\|u(t)\|_1^{(N+2b\theta)/N}\geq\|u(T)\|_1^{b+q}\int_t^T(s-t)^{b\theta}\,ds=\frac{1}{b\theta+1}(T-t)^{b\theta+1}\|u(T)\|_1^{b+q},
$$
which can be written in an equivalent form as
\begin{equation}\label{interm8}
\|u(T)\|_1\leq C(T-t)^{-(b\theta+1)/(b+q)}\|u(t)\|_1^{(N+2b\theta)/N(b+q)},
\end{equation}
which holds true for any $t\in(0,T)$. Let us notice that, when $m>1$,
\begin{equation*}
\begin{split}
N+2b\theta-N(b+q)&=N(1-q)-\frac{N^2(m-1)}{N(m-1)+2}b\\
&=\frac{N}{N(m-1)+2}\left[(1-q)(N(m-1)+2)-N(m-1)b\right]>0,
\end{split}
\end{equation*}
since we deduce from \eqref{interm3} that
$$
b<1-q+\frac{\sigma^*}{N}=\frac{(1-q)(N(m-1)+2)}{N(m-1)}.
$$
The adaptation for $m=1$ is immediate as then $\theta=N/2$ and the terms involving $b$ just cancel in the previous calculation. We thus have $N+2b\theta>N(b+q)$ and it follows from \cite[Lemme~4.1]{St65} that $u$ vanishes in finite time.
\end{proof}

\section{Proof of Theorem \ref{th.2}}\label{sec.th2}

Throughout this section, we fix $m>1$, $q\in(0,1)$ and $\sigma\geq\sigma^*$. Before starting the proof of Theorem~\ref{th.2}, we need a preparatory result on the availability of suitable stationary supersolutions to Eq.~\eqref{eq1}.
\begin{lemma}\label{lem.super}
Given $a\geq(\sigma+2)/(m-q)$ and $A>0$, $R>0$ satisfying~\eqref{interm1}, the function
$$
S_{a,A}(x):=A|x|^a, \qquad x\in\real^N,
$$
is a supersolution to Eq.~\eqref{eq1} on $(0,\infty)\times B(0,R)$.
\end{lemma}
\begin{proof}
Setting $l:=a(m-q)-\sigma-2\geq0$, and fixing some $x\in B(0,R)$, we obtain by direct calculation that
\begin{equation*}
\begin{split}
\partial_t S_{a,A}(x) - \Delta S_{a,A}^m(x) + |x|^{\sigma}S_{a,A}^q(x) &=-am(am+N-2)A^m|x|^{am-2}+A^q|x|^{\sigma+aq}\\
&=A^q|x|^{\sigma+aq}\left[1-am(am+N-2)A^{m-q}|x|^{l}\right]\\
&\geq A^q|x|^{\sigma+aq}\left[1-am(am+N-2)A^{m-q}R^l\right]\geq0,
\end{split}
\end{equation*}
where the last inequality follows from~\eqref{interm1}.
\end{proof}
We are now in a position to complete the proof of Theorem~\ref{th.2}.
\begin{proof}[Proof of Theorem~\ref{th.2}]
It follows from \eqref{cond.u0a} that
$$
u_0(x)\leq S_{a,A}(x)=A|x|^a, \qquad {\rm for \ any} \ x\in B(0,R),
$$
and from \eqref{cond.u0b} and \eqref{wp0} that
$$
u(t,x)\leq\|u_0\|_{\infty}\leq S_{a,A}(x), \qquad (t,x)\in(0,\infty)\times\partial B(0,R).
$$
Noticing that
$$
\frac{\sigma}{1-q}\geq\frac{\sigma+2}{m-q}, \qquad {\rm for} \ \sigma\geq\sigma^*
$$
and recalling that $A$ and $R$ satisfy~\eqref{interm1}, Lemma~\ref{lem.super} and the comparison principle in Theorem~\ref{th.wp} entail that
$$
u(t,x)\leq S_{a,A} = A|x|^a, \qquad (t,x)\in(0,\infty)\times B(0,R).
$$
On the one hand, the previous inequality implies that, for $(t,x)\in (0,\infty)\times B(0,R)$,
\begin{equation}\label{th2.int}
|x|^{\sigma}u^q(t,x)\geq\left(\frac{u(t,x)}{A}\right)^{\sigma/a}u^q(t,x)=A^{-\sigma/a}u^{(aq+\sigma)/a}(t,x).
\end{equation}
On the other hand, for $(t,x)\in (0,\infty)\times(\real^N\setminus B(0,R))$, we have $|x|\geq R$, which gives, along with~\eqref{wp0},
\begin{equation}\label{th2.ext}
|x|^{\sigma}u^q(t,x)=|x|^{\sigma}\frac{\|u_0\|_{\infty}^{\sigma/a}}{\|u_0\|_{\infty}^{\sigma/a}}u^q(t,x)\geq\frac{R^{\sigma}}{\|u_0\|_{\infty}^{\sigma/a}}u^{(aq+\sigma)/a}(t,x).
\end{equation}
Introducing
$$
B:=\min\{A^{-\sigma/a},R^{\sigma}\|u_0\|_{\infty}^{-\sigma/a}\},
$$
we obtain from~\eqref{th2.int} and~\eqref{th2.ext} that
$$
|x|^{\sigma}u^q(t,x)\geq Bu^{(aq+\sigma)/a}(t,x), \qquad (t,x)\in(0,\infty)\times\real^N.
$$
Therefore, the comparison principle gives that $u\leq v$ in $(0,\infty)\times\real^N$, where $v$ is the solution to
$$
\partial_t v-\Delta v^m+Bv^{(aq+\sigma)/a}=0, \qquad {\rm in} \ (0,\infty)\times\real^N
$$
with the same initial condition $v(0)=u_0$ in $\real^N$. Since $(aq+\sigma)/a<1$ by the choice of $a$, $v$ vanishes in finite time and so does $u$.
\end{proof}

We provide more precise examples of initial data for which finite time extinction holds true when $\sigma\ge \sigma^*$ in the following consequence of Theorem~\ref{th.2}.
\begin{corollary}\label{cor.2}
Let $m>1$, $0<q<1$, $\sigma\geq\sigma^*$ and consider $a>\sigma/(1-q)$ and $u_0\in L^{\infty}_{+}(\real^N)$ such that
\begin{equation}
u_0(x)\leq A_0|x|^a, \qquad x\in\real^N, \label{zz}
\end{equation}
for some $A_0>0$. Then there exists $M>0$ depending only on $m$, $a$, $N$, $q$, $\sigma$ and $A_0$ such that, if $\|u_0\|_{\infty}\leq M$, then the solution $u$ to the Cauchy problem~\eqref{eq1}-\eqref{init.cond} vanishes in finite time.
\end{corollary}
\begin{proof}
We set
$$
K:=\left[am(am+N-2)\right]^{1/(m-q)}, \qquad l:=a(m-q)-\sigma-2>0,
$$
and choose
$$
A=A_0, \qquad R=(KA_0)^{-(m-q)/l}, \qquad M=A_0R^a.
$$
Pick $u_0\in L_+^\infty(\real^N)$ satisfying~\eqref{zz} and $\|u_0\|_\infty\le M$. Then $A$ and $R$ satisfy~\eqref{interm1} and $u_0$ satisfies~\eqref{cond.u0}, so that we are in the hypothesis of Theorem~\ref{th.2}. An application of this theorem ends the proof.
\end{proof}

\bigskip

\noindent \textbf{Acknowledgements.} The authors are partially supported by the Spanish project PID2020-115273GB-I00. R. G. I. wants to thank for the hospitality and support of Institut de Math\'ematiques de Toulouse where part of this work has been done.

\bibliographystyle{plain}

\end{document}